\documentclass[a4paper,10pt]{amsart}
\usepackage[utf8]{inputenc}
\usepackage{amsmath}
\usepackage{amssymb}
\usepackage{amsthm}
\usepackage{mathrsfs}
\usepackage{mathtools}
\usepackage{hyperref}
\usepackage{enumerate}

\newcommand{\R}{\ensuremath{\mathbb{R}}}
\newcommand{\N}{\ensuremath{\mathbb{N}}}

\newcommand{\C}{C^1(K)}
\newcommand{\Cr}[1][d]{C^1(\R^{#1} | K)}
\newcommand{\Cro}{\ensuremath{{C}^1(\R | K)}}
\newcommand{\J}{\ensuremath{\mathcal{J}^1(K)}}
\newcommand{\E}{\ensuremath{\mathscr{E}^1(K)}}

\DeclareMathOperator{\supp}{supp}

\DeclareMathOperator{\dive}{div}
\DeclareMathOperator{\inte}{int}
\newcommand{\vol}{\ensuremath{\text{vol}}}

\newcommand{\Dd}{\mathscr{D}}
\newcommand{\eps}{\varepsilon}
\newcommand{\<}{\langle}
\renewcommand{\>}{\rangle}

\theoremstyle{definition}
\newtheorem{Def}{Definition}[section]
\newtheorem{Ex}[Def]{Example}
\newtheorem{Rmk}[Def]{Remark}

\theoremstyle{plain}
\newtheorem{Prop}[Def]{Proposition}
\newtheorem{Thm}[Def]{Theorem}

\title{Continuously differentiable functions on compact sets}

\author{Leonhard Frerick}
\address{Universit\"at Trier, FB IV  Mathematik, D-54286 Trier, Germany}
\email{frerick@uni-trier.de}
\author{Laurent Loosveldt}
\address{Universit\'e de Li\`ege, D\'epartement de math\'ematique -- zone Polytech 1, 12 all\'ee de la D\'ecouverte, B\^at. B37, B-4000 Li\`ege, Belgium}
\email{L.Loosveldt@uliege.be}
\author{Jochen Wengenroth}
\address{Universit\"at Trier, FB IV  Mathematik, D-54286 Trier, Germany}
\email{wengenroth@uni-trier.de}

\begin{document}

\begin{abstract}
We consider the space $\C$ of real-valued continuously differentiable functions on a compact set $K\subseteq \R^d$. We characterize the completeness of this space and 
prove that the restriction space $C^1(\R^d|K)=\{f|_K: f\in C^1(\R^d)\}$ is always dense in $\C$. The space $\C$ is then compared with other spaces of differentiable 
functions on compact sets.
\end{abstract}

\maketitle
\section{Introduction}
In most analysis textbooks differentiability is only treated for functions on open domains and, if needed, e.g., for the divergence theorem, an ad hoc 
generalization for functions on compact sets is given. We propose instead to  define differentiability on arbitrary sets as the usual affine-linear 
approximability -- the price one has to pay is then the definite article: Instead of {\it the} derivative there can be many. We will only consider compact  
domains in order to have a natural norm on our space. The results easily extended to $\sigma$-compact (and, in particular, closed) sets.

 An $\R^n$-valued function $f$ on a compact set $K\subseteq \R^d$ is said to belong $C^1(K,\R^n)$ if there exits a continuous function $df$ on $K$ 
 with values in the linear maps from $\R^d$ to $\R^n$ such that, for all $x \in K$,
 \begin{align} \label{differ}
 \lim_{\substack{y \to x\\ y \in K}} \frac{f(y)-f(x) - df(x) ( y-x )}{|y-x|} =0, 
 \end{align}
where $|\cdot|$ is the euclidean norm. For $n=1$ we often identify $\R^d$ with its dual
 and write $\langle \cdot , \cdot \rangle $ for the evaluation which is then the scalar product. 
 Questions about $C^1(K,\R^n)$ easily reduce to the case  $\C=C^1(K,\R)$.

Of course, equality (\ref{differ}) means that $df$ is a continuous (Fr\'echet) derivative of $f$ on $K$. As in the case of open domains,  every $f\in \C$ is 
continuous and we have the chain rule: For all (continuous) derivatives $df$ of $f$ on $K$ and $dg$ of $g$ on $f(K)$ the map $x \mapsto dg(f(x))\circ df(x)$ 
is a (continuous) derivative of $g\circ f$ on $K$.

    In general, a derivative need not be unique. For this reason, a good tool to study $\C$ is the jet space
\begin{align*}
\J = \{ (f,df): df \text{ is a continuous derivative of $f$ on $K$} \}
\end{align*}
endowed with the norm
\begin{align*}
\|(f,df)\|_{\mathcal{J}^1(K)} = \|f\|_K + \|df \|_K,
\end{align*}  
where $\|\cdot\|_K$ is the uniform norm on $K$ and $|df(x)|=\sup\{|df(x)(v)|:|v|\le 1\}$. For the projection $\pi(f,df)=f$ we have
$\C=\pi(\J)$, and we equip $\C$ with the quotient norm, i.e.,
\begin{align*}
\| f \|_{\C} = \|f\|_K + \inf \{ \|df\|_K: df  \text{ is a continuous derivative of $f$ on $K$} \}.
\end{align*}

It seems that the space $\C$ did not get much attention in the literature. This is in sharp
contrast to the ``restriction space'' $\Cr=\{f|_K: f\in C^1(\R^d)\}$. Obviously, the inclusion $\Cr \subseteq \C$ holds but it is well-known that, in general, 
it is strict. Simple examples are domains with inward directed cusps like
\begin{align*}
K=\{(x,y) \in [-1,1]^2 : |y| \geq e^{-1/x} \text{ for }  x > 0 \}.
\end{align*}
The function $f(x,y)=e^{-1/2x}$ for $x,y>0$ and $f(x,y)=0$ elsewhere is  in $\C$ but it is not the restriction of a $C^1$-function on $\R^2$ because is 
is not Lipschitz continuous near the origin.
 
In a famous paper from 1934 \cite{Whi34}, Whitney proved that $\Cr= \pi(\E)$ where $\E$ is the spaces of jets $(f,df)$ for which the limit (\ref{differ}) 
is uniform in $x \in K$. Moreover, $\E$ endowed with the norm
\[
\|(f,df)\|_{\E} = \|(f,df)\|_{\J}+ \sup\left\{ \frac{|f(y)-f(x)|}{|y-x|}: x,y\in K, y \neq x\right\}
\]
is a Banach space. Thus, $\Cr$ equipped with the quotient norm \linebreak $\|\cdot \|_{\Cr}$ inherited from $\|\cdot \|_{\E}$ is also a Banach space.

Since their introduction, Whitney jets (also of higher orders) have been widely studied, in particular in the context of extension operators \cite{Fef05,Fre07,Fre11,Fre16}. 
Generalizations of them have been defined in various contexts such as Baire functions \cite{Koc12}, holomorphic functions \cite{Bru03} Sobolev spaces \cite{Zob98, Zob99}, 
so-called $C^{m,\omega}(\R^d)$ spaces \cite{Fef05a} or (generalized) H\"older spaces \cite{Loo20}.

In this paper, we prove that $\E$ is always a dense subset of $\J$. The density of $\Cr$ in $\C$ is then an immediate consequence. Together with a characterization 
of the completeness of $(\C,\|\cdot \|_{\C})$, this leads to a simple geometric criterion for the equality $\C=\Cr$ as Banach spaces. In the one-dimensional case, 
we also give a characterization of the mere algebraic equality.

If the compact set $K$ is topologically regular, i.e., the closure of its interior, another common way to define  differentiability is the space
\begin{align*}
{C}^1_{\inte}(K)= \{ f \in C^1(\mathring{K}): f \text{ and } df \text{ extend continously to } K \},
\end{align*}
see for instance \cite{Fol95,Zie89}.

In this situation, the derivative of a continuously differentiable function on $K$ is uniquely determined by the function, which means that 
the projection $\pi$ is injective on $\J$ and therefore $\C$ and $\J$ as well as $\Cr$ and $\E$, respectively, can be identified.

Equipped with the norm $\|f\|_K + \|df\|_K$, it is clear that ${C}^1_{\inte}(K)$ is always a Banach space. Despite this nice aspect we will see by 
an example of Sauter \cite{Sau} that ${C}^1_{\inte}(K)$ has dramatic drawbacks: The chain rule fails in this setting and compositions of $C^1_{\inte}(K)$-functions 
need not be differentiable. 

We will present some results about equalities between $C^1_{\inte}(K)$, $\Cr$ and $\C$, giving an echo to the so-called ``Whitney conjecture'' ( \cite{Zob99,Whi34c}).

\bigskip 

The paper is organized as follows. In section 2, we start with some more or less standard facts about rectifiable paths and 
integration along them to establish the fundamental 
theorem of calculus for $\C$-functions, and we present the above mentioned example of Sauter of $C^1_{\inte}$-functions where this result fails. In Section 3 we 
characterize the completeness of 
$\C$ by a simple geometric condition, and in Section 4, we prove the density of $\Cr$, which relies on very deep results of Smirnov \cite{Smi94}. 
In Section 5, we compare the spaces $\Cr$, $\C$ and $C^1_{\inte}(K)$ and finally, we give some complementary specific results for compact subsets of $\R$.

\section{Path integrals}

A function $f\in \C$ need not be Lipschitz continuous because segments with endpoints in $K$, to which one would like to apply the mean value theorem, 
need not be contained in $K$. Instead of segments one then has to consider rectifiable paths in $K$, i.e., continuous functions $\gamma:[a,b]\to K$ such that the length
\begin{align*}
L(\gamma ) = \sup \left\{ \sum_{j=1}^{n} |\gamma(t_j)-\gamma(t_{j-1})|: a=t_0< \cdots < t_n =b \right \} 
\end{align*}
is finite. The function $\ell(t)=L(\gamma|_{[a,t]})$ is then continuous: Given $\varepsilon>0$ and a partition such that the length of the corresponding polygon is bigger than
$L(\gamma)-\varepsilon$ every interval $[r,s]$ lying between two consecutive points of the partition satisfies $\ell(s)-\ell(r)=L(\gamma_{[r,s]})\le |\gamma(s)-\gamma(r)|+\varepsilon$.
For the minimal length of the subintervals of the partition one then easily gets the required continuity estimate.

\begin{Prop}[Mean value inequality]
Let $f \in \C$ and $x, y \in K$. If $df$ is a derivative of $f$ on $K$ and if $x$ and $y$ are joined by a rectifiable path $\gamma  :  [a,b] \to K$, then
\begin{align}\label{tafine}
|f(y)-f(x)| \leq L(\gamma) \sup \{ |d f(z)|: z \in \gamma([a,b]) \}.
\end{align}
\end{Prop}

\begin{proof} We essentially repeat H\"ormander's proof \cite[theorem 1.1.1]{Hor90}.
For each $c>\sup \{ |d f(z)|: z \in \gamma([a,b]) \}$ the set 
\(
  T=\{t\in[a,b]: |f(\gamma(t))-f(x)|\le c\ell(t)\}
\)  
  is non-empty and closed because of the continuity of $f\circ\gamma$ and 
$\ell$, hence is has a largest element $t\in[a,b]$. If $t$ were different from $b$, the differentiability of $f$ at $z=\gamma(t)$ gives a neighbourhood $U$ of $z$ such that
\[
  |f(z)-f(w)|\le |f(z)-f(w)-df(z)(z-w)|+|df(z)(z-w)| \le c|z-w|
\]
for all $w\in U$. By the continuity of $\gamma$ we find $s>t$  with $\gamma(s)\in U$ so that
\[
  |f(\gamma(s))-f(x)| \le |f(\gamma(s))-f(\gamma(t))|+ c\ell(t) \le c|\gamma(s)-\gamma(t)|+c\ell(t)\le c\ell(s),
\]
contradicting the maximality of $t$.
\end{proof}

The mean value inequality does not use the continuity of a derivative and has the usual consequences. 
For example, if $df=0$ is a derivative of $f$ and $K$ is {\it rectifiably pathwise connected} (a certainly self-explaining notion) then $f$ is constant. 

Our next aim is to show that a continuous
derivative integrates back to the function along rectifiable paths. We first recall the relevant notions. If $F:K\to \R^d$ is continuous
and $\gamma$ is a rectifiable path in $K$ we define the path integral $\int_\gamma F$ as the limit of Riemann-Stieltjes sums
\[
  \sum_{j=1}^n \< F(\gamma(\tau_j)), \gamma(t_j)-\gamma(t_{j-1})\>
\]
where $a=t_0<\ldots<t_n=b$ are partitions with $\max\{t_j-t_{j-1}:1\le j\le n\}\to 0$ and $t_{j-1 }\le \tau_j\le t_j$. 
The existence of the limit is seen from an appropriate Cauchy condition (or by using the better known one-dimensional case where rectifiable paths
are usually called functions of bounded variation). If $\gamma$ is even absolutely continuous, i.e., there is a Lebesgue integrable $\dot\gamma:[a,b]\to\R^d$ with
$\gamma(\beta)-\gamma(\alpha)=\int_\alpha^\beta \dot\gamma(t)dt$ for all $\alpha\le\beta$, one gets from the uniform continuity of $F\circ \gamma$ the familiar representation
\[
  \int_\gamma F=\int_a^b \<F(\gamma(t)),\dot\gamma(t)\>dt.
\]
If $\gamma$ is even continuously differentiable and $F=df$ for a function $f\in\C$, the integrand in the last formula is the derivative of $f\circ\gamma$ (by the chain rule)
and the fundamental theorem of calculus gives $\int_\gamma df= f(\gamma(b))-f(\gamma(a))$. Since continuous differentiability of $\gamma$ is a not a realistic assumption in our considerations 
(interesting phenomena typically occur for quite rough compact sets $K$) we need a more general version:

\begin{Thm}[Fundamental theorem of calculus]\label{thm:pathint}
For each $f \in \C$ with a continuous derivative $df$ and each rectifiable $\gamma:  [a,b] \to K$ we have
\begin{align}\label{pathint}
\int_\gamma df = f(\gamma(b))-f(\gamma(a)).
\end{align}
\end{Thm} 

\begin{proof}
  Given a partition $a=t_0<\ldots<t_n=b$ and a fixed $j\in\{1,\ldots,n\}$ we set $z=\gamma(t_j)$ and apply the mean value inequality to the function
  \[
    g(x)=f(x)-f(z)-\<df(z),x-z\>
  \]
on $\gamma([t_{j-1},t_j])$. Since $dg(x)=df(x) -df(z)$ is a derivative of $g$ we obtain
\begin{align*}
 & \left|f(\gamma(t_j)) - f(\gamma(t_{j-1}))-\<df(\gamma(t_j)),\gamma(t_j)-\gamma(t_{j-1})\>\right| \\
 & =  |g(\gamma(t_j))-g(z)| \le L(\gamma|_{[t_{j-1},t_j]}) \sup\{|df(\gamma(t))-df(\gamma(t_{j-1})|: t\in [t_{j-1},t_j]\}.
\end{align*}
The uniform continuity of $df\circ \gamma$ yields that this supremum is small whenever the partition is fine enough. The theorem then follows by writing
$f(\gamma(b))-f(\gamma(a))$ as a telescoping sum and inserting these estimates together with the obvious additivity of the length.
\end{proof}

Below, we will need a slightly more general version of the fundamental theorem: The formula $\int_\gamma df= f\circ\gamma|_a^b$ holds if $f$ and $df$
are continuous on $K$ and $df(x)$ is a derivative of $f$ at $x$ for all but finitely many $x\in \gamma([a,b])$.

Indeed, if only the endpoints $\gamma(a)$ and $\gamma(b)$ are exceptional, this follows from a simple limiting argument, the general case is then obtained by decomposing
the integral $\int_\gamma df$ into a sum.

Once in this article, we will have to find a rectifiably path by using the Arzel\'a-Ascoli theorem. 
It is then essential to have a ``tame'' parametrization which we explain briefly, more details can be found, e.g., in \cite{Haj03}.
Given a continuous $\gamma:[a,b]\to\R^d$ with length $L=L(\gamma)$ and length function $\ell(t)=L(\gamma|_{[a,t]})$ the function
$\alpha(s)=\inf\{t\in[a,b]: \ell(t)\ge s\}$ is again increasing but not necessarily continuous, it jumps over the
intervals where $\ell$ is constant. Nevertheless, $\tilde\gamma=\gamma\circ\alpha: [0,L]\to \R^d$ is a continuous path with $\tilde\gamma([0,L])=\gamma([a,b])$ 
such that all path integrals along $\gamma$ and $\tilde\gamma$ coincide and such that $L(\tilde\gamma|_{[0,t]})=t$ for all $t\in[0,L]$, in particular,
$\tilde\gamma$ is Lipschitz with constant $1$. This path $\tilde\gamma$ is called the parametrization of $\gamma$ by arclength.

If $\{\gamma_i:i\in I\}$ is a family of curves with equal length, it then follows that $\{\tilde\gamma_i:i\in I\}$ is equicontinuous.
Moreover, Rademacher's theorem implies that $\tilde\gamma$ is almost everywhere differentiable and absolutely continuous.

 \medskip
 
 We have seen that the behaviour of functions $f\in\C$ is essentially as in the case of open domains. We will now present Sauter's example \cite{Sau} showing 
that this not the case for $f\in C^1_{\inte}(K)$.

Let $C$ be the ternary Cantor set and $U$ its complement in $(0,1)$. The open set $\Omega$ is constructed from $U \times (0,1)$ by removing 
disjoints balls $(B_{j})_{j \in \N}$ that accumulate at $C \times [0,1]$ such that the sum of the diameters is $<1/4$.
This implies that there exist horizontal lines in $K=\overline{\Omega}$.

If $f$ is the Cantor function on $[0,1]$, we consider the function $F$ defined on $K$ by $F(x,y)=f(x)$. We have $F \in C^1_{\inte}(K)$ 
because it is continuous and  $dF=0$ on $\Omega=\mathring{K}$, as $f$ is locally constant on $U$. 
If now $\gamma$ is a path parametrizing one of the horizontal lines crossing $K$, we have
\begin{align*}
\int_{\gamma} dF = 0 \text{ while }  F(\gamma(1))-F(\gamma(0))=f(1)-f(0)=1.
\end{align*}
This proves $f\notin \C$. This example also reveals the catastrophy that compositions (namely $F \circ \gamma$) of $C^1_{\inte}$-functions 
need not be $C^1_{\inte}$.

\section{Completeness}

We study here the completeness of $(\C,\|\cdot \|_{\C})$ and $(\J,\|\cdot \|_{\J})$. We show that, if $K$ has infinitely 
many connected components, then these spaces are not complete. In contrast, if $K$  has finitely many connected components, 
the completness of both spaces is characterized by a pointwise geometric condition whose uniform version goes back to Whitney in \cite{Whi34c}. 
It is interesting to note that this characterization is conjectured in \cite{Dal10} in the context of complex differentiability. 

First we consider the case of compact sets with infinitely many connected components. This is similar to \cite[Theorem 2.3]{Bla05}.

\begin{Prop}\label{Prop:incic}
If $K$ is a compact set with infinitely many connected components, then $(\C,\|\cdot \|_{\C})$ is incomplete.
\end{Prop}
\begin{proof}
We can partition $S_0=K$ into two non-empty, disjoint open subsets $S_1$ and $K_1$ such that $S_1$ has infinitely many connected components. 
Iterating this procedure we obtain a sequence $(K_j)_{j \in \N}$ of pairwise disjoints non-empty closed and open subsets of $K$. 

We fix $x_j \in K_j$ and, by compactness and passing to a subsequence, we can assume that $x_{j}$ convergeges in $ K$. The limit $x_0$ cannot belong to any 
$K_{j}$ because they are open and pairwise disjoint.

We consider the functions $f_n:K\to\R$ defined by $f_n(x)=|x_j-x_0|$ for $x\in K_j$ with $1\le j\le n$ and $f_n(x)=0$, else.
These functions are locally constant and hence $f_n \in \C$. It is easy to check that $(f_n)_{n \in \N}$ is a Cauchy sequence in $(\C,\|\cdot \|_{\C})$. 
The only possible limit is the function $f(x)=|x_j-x_0|$ for $x\in K_j$ and $j\in\N$ and $f(x)=0$, else.
But, for all $j \in \N$, we have
\begin{align*}
\frac{|f(x_{j})-f(x_0)|}{|x_{j}-x_0|}=1,
\end{align*}
and since $df_n=0$ this shows that $f$ cannot be the limit in $\C$.
\end{proof}

A set $K \subseteq \R^d$ is called \textit{Whitney regular} if there exists $C>0$ such that any two points $x,y \in K$ can be joined by a rectifiable path 
in $K$ of length bounded by $C |x-y|$. 

We say that $K$ is \textit{pointwise Whitney regular} if, for every $x \in K$, there are a neighbourhood $V_x$  of $x$  and $C_x>0$ such that any $y \in V_x$ is 
joined to $x$ by a rectifiable path in $K$ of length bounded by $C_x |x-y|$.

The inward cusp mentioned in the introduction distinguishes these two notions. If $K$ is {\it geodesically bounded} 
(i.e., any two points can be joined by a curve of length bounded by a fixed constant) one can take $V_x=K$ in the definition so that the crucial difference 
is then the non-uniformity of the constants $C_x$.

\begin{Prop}\label{prop:whba}
If $K$ is a pointwise Whitney regular compact set, then the space $(\J,\|\cdot \|_{\J})$ is complete.
\end{Prop}

\begin{proof}
For a Cauchy sequence $((f_j,df_j))_{j \in \N}$ in $\J$ we get from the completeness of $C(K)$ uniform limits $f$ and $df$ and we only have to show that $df$ 
is a derivative of $f$.

Given $x \in K$ and a path $\gamma$ from $x$ to $y$ of length $L(\gamma) \leq C_x |x-y|$, the formula in the fundamental theorem of calculus immediately 
extends from $f_j$ and $df_j$ to the limits and thus gives
\begin{align*}
f(y)-f(x)-\langle df(x),y-x \rangle = \int_\gamma (df-df(x)).
\end{align*}
The continuity of $df$ and the bound on $L(\gamma)$ then easily imply the desired differentiability.
\end{proof}

To obtain the converse of this simple result we first apply the uniform boundedness principle to show that the completeness of $(\C,\|\cdot \|_{\C})$ is equivalent 
to some bounds for the difference quotient of a function $f \in \C$. This is the same as in the case of complex differentiability \cite{Hon99,Bla05}.

\begin{Prop}
The following assertions are equivalent:
\begin{enumerate}[(a)]
\item The space $(\J,\|\cdot \|_{\J})$ is a Banach space.
\item The space $(\C,\|\cdot \|_{\C})$ is a Banach space.
\item For every $x \in K$, there exists $C_x>0$ such that for all $f \in \C$ and $y\in K\setminus \{x\}$
\begin{align} \label{prop:baest}
 \frac{|f(y)-f(x)|}{|y-x|} \leq C_x \| f \|_{\C}.
\end{align} 
\end{enumerate}
\end{Prop}

\begin{proof}
That (a) implies (b) is a standard fact from Banach space theory.
Let us show that the second assertion implies the third. For fixed $x \in K$ and each $y \in K\setminus\{x\}$ we define a linear and continuous functional on $\C$ by
\begin{align*}
\Phi_y(f)=\frac{f(y)-f(x)}{|y-x|}.
\end{align*}
For fixed $f \in \C$, we get a bound for $\sup_{y \in K \setminus \{x\}} |\Phi_y(f)|$ because of the differentiability at $x$.

The Banach-Steinhaus theorem thus gives
\begin{align*}
C_x=\sup \{|\Phi_y(f)|: \|f\|_{\C} \leq 1, y\in K\setminus\{x\} \} <\infty.
\end{align*}

Now we assume that inequality (\ref{prop:baest}) holds and show that $(\J,\|\cdot \|_{\J})$ is complete. For a Cauchy sequence $((f_j,df_j))_{j \in \N}$ 
in $\J$  we have uniform limits $f$ and $df$. In particular, for all $\varepsilon >0$, $x \in K$, and $p<q$ big enough, we have
\begin{align*}
\|f_p - f_q \|_{\C} \leq \|f_p - f_q \|_{\J} \leq \frac{\varepsilon}{4 C_x} \quad \text{and} \quad  \|df_p - df \|_K < \frac{\varepsilon}{4}.
\end{align*}
Now, there exists $\delta>0$ such that, for all $y \in B(x,\delta) \setminus \{x \}$,
\begin{align*}
B=\frac{|f_p(y)-f_p(x)-\langle df_p(x), y-x \rangle|}{|y-x|} < \frac{\varepsilon}{4}.
\end{align*}
Finally, for all such $y$, if $q$ is large enough,
\begin{align*}
A = \frac{|(f(y)-f_q(y))-(f(x)-f_q(x))|}{|y-x|} < \frac{\varepsilon}{4}
\end{align*}
and
\begin{align*}
&\frac{|f(y)-f(x)-\langle df(x),y-x \rangle |}{|x-y|} \\&\leq A +  \frac{|(f_p(y)-f_q(y))-(f_p(x)-f_q(x))|}{|y-x|} + B + |df_p(x)-df(x) | \\& < \varepsilon
\end{align*}
which shows that $df$ is a derivative of $f$ on $K$.
\end{proof}

Next we show that, for connected sets $K$, inequality (\ref{prop:baest}) implies pointwise regularity. This is 
a simple adaptation of a result in \cite[theorem 2.3.9]{Hor90}, we repeat the proof for the sake of completeness.

\begin{Prop} \label{Prop:estimpwhi}
Let $K$ be a compact connected set. If, for any $x \in K$, there exists $C_x>0$ such that for all $f \in \C$ and $y\in K\setminus \{x\}$ we have
\begin{align}\label{estwh1}
 \frac{|f(y)-f(x)|}{|y-x|} \leq C_x \| f \|_{\C},
\end{align}
then $K$ is pointwise Whitney regular.
\end{Prop}

\begin{proof}
For any $\varepsilon>0$,
\begin{align*}
K_{\varepsilon} = \{x \in \R^d \, : \, \inf_{y \in K } |x-y| < \varepsilon \}
\end{align*}
is an open connected neighbourhood  of $K$. Let us fix $x \in K$ and define the function $d_\varepsilon$ on $K_{2\varepsilon}$ by
\begin{align*}
d_\varepsilon(y) = \inf \{L(\gamma): \gamma \text{ rectifiable path from $x$ to $y$ in $K_{2\varepsilon}$} \}.
\end{align*}
Then, for fixed $y_0 \in K$, we set $u_\varepsilon(y) = \min \{d_\varepsilon (y), d_\varepsilon (y_0) \}$. If $y$ and $y'$ are close enough in $K_{2\varepsilon}$, 
we have
\begin{align} \label{estwh2}
|u_\varepsilon (y) - u_\varepsilon (y')| \leq |y-y'|,
\end{align}
as any rectifiable path from $x$ to $y$ prolongs by the segment between $y$ and $y'$ to a rectifiable path from $x$ to $y'$. 

If $\phi$ is a positive smooth function with support in $B(0,\varepsilon)$ and integral $1$, the convolution $u_\varepsilon \ast \phi$, 
defined in $K_\varepsilon$, is a smooth function for which $|d(u_\varepsilon \ast \phi)| \leq 1$ on $K$, because of inequality (\ref{estwh2}). 
Then, from (\ref{estwh1}), we have
\begin{align*}
|(u_\varepsilon \ast \phi)(x)-(u_\varepsilon \ast \phi)(y_0)| \leq C_x (d_\varepsilon (y_0) + 1) |x-y_0|
\end{align*}
which gives us, passing to the limit $\supp (\phi) \to \{0 \}$,
\begin{align*}
d_\varepsilon (y_0) \leq  C_x (d_\varepsilon (y_0) + 1) |x-y_0|.
\end{align*}
For $y_0 \in B(x,\frac{1}{2 C_x})\cap K$, this implies $d_\varepsilon(y_0) \le 1$ and thus $d_\varepsilon(y_0)\le 2C_x|x-y_0|$. Hence, there exists a rectifiable path from $x$ to $y_0$ in $K_{2\varepsilon}$ 
of length bounded by $2 C_x  |x-y_0| + \varepsilon$. Using the parametrization by arc length gives an  equicontinuous family of paths
and the conclusion follows from the Arzel\'a -Ascoli theorem. 
\end{proof}

\begin{Rmk} \label{Rem:estimpwhi}
If the constant $C_x$ in previous proposition is uniform with respect to $x \in K$, then inequality (\ref{estwh2}) is equivalent to the Whitney regularity of $K$, 
as stated in H\"ormander's book.
\end{Rmk}

Gathering all the results of this section we have the following characterization of the completeness of $(\C,\| \cdot \|_{\C} )$.

\begin{Thm}\label{thm:compl}
$(\C,\| \cdot \|_{\C} )$ is complete if and only if $K$ has finitely many components which are pointwise Whitney regular.
\end{Thm}

\begin{Rmk}
In this pointwise Whitney regular situation, the jet space $\J$ can be described as a space of continuous ``circulation free vector fields" $F$ on $K$, 
i.e., vector fields $F$ for which $\int_\gamma F=0$ for all closed rectifiable paths $\gamma$ in $K$. More precisely, if $(f,df) \in \J$, 
the fundamental theorem of calculus implies that $df$ is circulation free and if $F$ is circulation free and continuous we can define, for some fixed $x_0 \in K$, 
for all $x \in K$
\begin{align*}
f(x) = \int_\gamma F
\end{align*}
where $\gamma$ is a path in $K$ from $x_0$ to $x$. This definition makes sense as $F$ is circulation free and $F$ is a continuous derivative of $f$ on $K$, 
by a similar argument as in the proof of proposition \ref{prop:whba}.
\end{Rmk}

\section{Density of restrictions}

In this section we will show that the space $\Cr$ of restrictions of continuously differentiable functions on $\R^d$ to $K$ is always dense in $\C$. 
As $\Dd(\R^d)$, the space of $C^\infty$-functions with compact support, is dense in $C^1 (\R^d)$, this is the same as the density of test functions in $\C$ and again, 
it is advantageous to consider 
this question on the level of jets, that is, we will show that
\begin{align*}
i:  \Dd(\R^d) \to \J, \, \varphi \mapsto (\varphi|_{K} , d\varphi|_{K}) 
\end{align*}
has dense range. 

For general $K$, all standard approximation procedures like convolution with smooth bump functions do not apply easily, and we will use the Hahn-Banach theorem instead.

A continuous linear functional $\Phi$ on $\J \subseteq C(K)^{d+1}$ is, by the Hahn-Banach and Riesz's representation theorem, given by signed measures $\mu, \mu_1,\cdots, \mu_d$ on $K$ via
\begin{align*}
\Phi(f,df) = \int f  d\mu + \sum_{j=1}^{d} \int d_j f  d\mu_j,
\end{align*}
where $d_jf$ are the components of $df$. If $\Phi$ vanishes on the image of $i$ we have, for all $\varphi\in\Dd(\R^d)$,
\begin{align*}
\int \varphi d\mu + \sum_{j=1}^{d} \int \partial_j \varphi  d\mu_j =0.
\end{align*}
For the distributional derivatives of the measures this means that
\begin{align*}
\mu = \sum_{j=1}^{d} \partial_j \mu_j = \dive (T)
\end{align*} 
where $T = (\mu_1,\ldots, \mu_d)$ is a vector field of measures or a {\it charge}.

Fortunately, such charges were throughly investigated by Smirnov in \cite{Smi94}. Roughly speaking, he proved a kind of 
Choquet representation of charges in terms of very simple ones induced by Lipschitz paths in $K$. If $\gamma :  [a,b] \to K$ is 
Lipschitz with a.e.\ derivative $\dot\gamma=(\dot\gamma_1,\ldots,\dot\gamma_d)$ and $F=(F_1,\ldots,F_d)$ is a continuous vector field we have, as noted in section 2,
\begin{align*}
\int_\gamma F = \int_a^b \langle F(\gamma(t)), \dot\gamma(t)\rangle  dt = \sum_{j=1}^{d} \int_a^b F_j (\gamma(t)) \dot\gamma_j(t)  dt.
\end{align*}
In order to see this as the action $\< T,F\>=\sum\limits_{j=1}^d \int F_jd\mu_j$ of a charge $T=(\mu_1,\ldots,\mu_d)$ we denote by
$\mu_j$ is the image (or push-forward) under $\gamma$ of the measure with density $\dot\gamma_j$ on $[a,b]$ so that $\int F_j(\gamma(t))\dot\gamma_j(t)dt=\int F_jd\mu_j$.
 For the charge $T_\gamma=(\mu_1,\ldots,\mu_d)$ we then have
\[
\< T_\gamma,F\>= \int_\gamma F.
\]

The fundamental theorem of calculus for $\varphi \in \mathcal{D}(\R^d)$ with derivative $d\varphi$ then gives
\begin{align*}
\dive (T_\gamma) (\varphi) = - \int_\gamma d\varphi = \varphi(\gamma(a))-\varphi(\gamma(b)) = (\delta_{\gamma(a)}-\delta_{\gamma(b)})(\varphi), \text{ that is}
\end{align*}
\[
\dive(T_\gamma)= \delta_{b(\gamma)}-\delta_{e(\gamma)}
\]
where $b(\gamma)$ and $e(\gamma)$ denote the beginning and the end of $\gamma$
(the change of signs comes from the minus sign in the definition of distributional derivatives).

To formulate Smirnov's results we write $\Gamma$ for the set of all Lipschitz paths in $\R^d$. Moreover, for a charge $T$ we denote by
\begin{align*}
\|T\|(E) = \sup \left\{ \sum_{j \in \N} |T(E_j)|: (E_j)_{j\in\N}  \text{ is a partition of $E$} \right\}
\end{align*}
the corresponding variation measure.

Given a set $\mathcal{S}$ of charges, a charge $T$ is said to decompose into elements of $\mathcal{S}$ if there is a finite, positive measure on $\nu$ on $\mathcal{S}$ 
(endowed with the Borel $\sigma$-algebra with respect to the weak topology induced by the evaluation 
$\< (\mu_1,\ldots,\mu_d), (\varphi_1,\ldots,\varphi_d) \>= \sum_{j=1}^d \int \varphi_j \, d\mu_j, \, \varphi_j \in \mathcal{D}(\R^d)$) such that
\begin{align*}
T = \int_\mathcal{S} R \  d\nu(R)  \text{ and }  \|T\| = \int_\mathcal{S} \|R \|  d\nu(R) 
\end{align*}
where these integrals are meant in the weak sense, i.e., $\langle T, \varphi \rangle = \int_\mathcal{S} \langle R, \varphi \rangle \, d\nu(R)$ 
for all $\varphi \in (\mathcal{D}(\R^d))^d$. By density and the continuity of charges with respect to the uniform norm, this extends to all $\varphi \in (C_c (\R^d))^d$, where $C_c (\R^d)$ is 
the space of continuous functions with compact support.

We can now state a consequence of Smirnov's results (theorem C of \cite{Smi94} is somewhat more precise than we need).

\begin{Thm} \label{Thm:Smi1}
Every charge $T$ with compact support such that $\dive(T)$ is a signed measure can be decomposed into elements of $\Gamma$, i.e., there is a positive finite measure $\nu$ on $\Gamma$ such that
\[
T=\int_\Gamma T_\gamma d\nu(\gamma) 
\text{
 and } \|T\|=\int_\Gamma \|T_\gamma\|d\nu(\gamma).
\]
\end{Thm}

The decomposition of the corresponding variation measures has the important consequence that the supports of $\nu$-almost all $T_\gamma$ are contained in the support of $T$ (where the supports are meant as the supports of signed measures which coincide with the supports of the corresponding distributions). After removing a set of $\nu$-measure $0$ we can thus assume that all paths involved in the decomposition of $T$ have values in the support of $T$. Using the definition of the distributional derivative we also obtain a decomposition of the divergences:
\[
\dive(T)=\int_\Gamma \dive(T_\gamma) d\nu(\gamma)=
\int_\Gamma \delta_{b(\gamma)} -\delta_{e(\gamma)} d\nu(\gamma).
\]
We are now prepared to state and prove the main result of this section.

\begin{Thm}\label{the:dense}
For each compact set $K$, the space $C^1(\R^d|K)$ is dense in $\C$.
\end{Thm}

\begin{proof}
We will show that $i:\mathscr D(\R^d)\to \J$, $\varphi\mapsto (\varphi|_K,d\varphi|_K)$ has dense range, the conclusion then follows by projecting onto the first components.

Let us consider $\Phi \in (C(K)^{d+1})'$ such that $\Phi$ vanishes on the range of 
$i$, by the Hahn-Banach theorem  it is enough to show that $\Phi|_{\J}=0$. 

As explained at the beginning of this section we get signed measures $\mu$ and $\mu_j$  on $K$ with
\begin{align*}
\Phi ((f,f_1,\cdots,f_d))=\int f  d\mu + \int f_1 d\mu_1 + \cdots + \int f_d  d\mu_d
\end{align*}
for all $ (f,f_1,\cdots,f_d) \in C(K)^{d+1}$,
and $T=(\mu_1 , \cdots , \mu_d)$ satisfies $\dive (T) =\mu$. We can thus apply theorem \ref{Thm:Smi1} and get a measure $\nu$ and $\mathcal{S} \subseteq \Gamma$ such that all paths in $\mathcal{S}$ have values in $K$ and
\begin{align*}
T = \int_{\mathcal{S}} T_\gamma  d\nu(\gamma).
\end{align*}
For $ (f,df)=(f,d_1 f,\ldots,d_d f) \in \J$ we extend all components to $C_c(\R^d)$ by Tietze's theorem and 
obtain from the fundamental theorem of calculus for $\C$-functions
\begin{align*}
& \int d_1 f d\mu_1 + \cdots + \int d_df  d\mu_d = \<T,df\>= \int_\mathcal{S} \langle T_\gamma,df\rangle  d\nu(\gamma) 
    \\
& =  \int_\mathcal{S} \delta_{e(\gamma)}(f)-\delta_{b(\gamma)}(f)  d\nu(\gamma) 
= - \dive (T) (f)
=-\int f  d\mu,
\end{align*}
which means that $\Phi|_{\J}=0$.
\end{proof}

The use of the Hahn-Banach theorem has the disadvantage of not giving any concrete approximations. Let us therefore very briefly mention two 
situations where they can be described explicitly.

A natural idea is to glue the local approximation given by the definition of differentiability together with a partition of unity. 
We decompose $\R^d$ into $d$-dimensional squares $Q_j$, choose points $x_j \in K \cap Q_j$ and a partition of unity $(\varphi_j)_j$ subordinated 
to slightly bigger squares with a fixed number of overlaps and bounds on the derivatives $|\partial_k \varphi_j| \leq C \vol(Q_j)^{-1}$ as,
e.g., in \cite[Thm.~1.4.6]{Hor90}. Then one expects
\begin{align*}
h(x)= \sum_j \varphi_{j}(x) \left(f(x_j)+\langle df(x_j),x-x_j \rangle\right)
\end{align*}
to be an approximation in $\C$ of a given $f$.

However, to estimate $\|df-dh\|_K$ by using theorem \ref{thm:pathint} requires enough curves in $K$ with uniform bounds on the length, i.e., that $K$ is Whitney regular.

\medskip

An even simpler approximation works  for compact sets which are (locally) starlike or, in the terminology of Feinstein, Lande and O'Farrell \cite{Fei96}
``locally radially self-absorbing". In the simplest case, we have $K \subseteq r \overset{\circ}{K}$ for some $r>1$. 
Given then $f \in \C$ one gets an approximation $h(x)=f(\frac{1}{r}x)$ on $r \overset{\circ}{K}$ for $r$ close to $1$ which one can 
multiply with a cut-off function which is $1$ near $K$ to get an approximation by functions in $C^1(\R^d)$. This ``blow up trick" can be 
localized with the aid of partition of unity.

\section{Comparison}

In this section, we compare the spaces $\Cr$, $\C$ and $C^1_{\inte}(K)$.

\begin{Thm} \label{Thm:E=C}
$\C=\Cr$ with equivalent norms if and only if $K$ has only finitely many components which are all Whitney regular.
\end{Thm}

\begin{proof}
Assuming the stated isomorphism of normed spaces we get that $\C$ is complete and proposition \ref{Prop:incic} implies that $K$ has only finitely many components. Moreover,
the equivalence of norms implies $\frac{|f(y)-f(x)|}{|y-x|} \leq C  \|df\|_{\C}$ for some constant so that remark \ref{Rem:estimpwhi} implies that each component is
Whitney regular.

For the other implication we first note that the global Whitney condition for each of the finitely many components implies, by the mean value inequality, the equivalence of 
the norms $\|\cdot\|_{\Cr}$ and $\|\cdot\|_{\C}$ on $\Cr$. This is thus a complete and hence closed subspace of $\C$ and, on the other hand, it is dense by theorem \ref{the:dense}.
\end{proof}

If we assume a priori the completeness of $\C$, i.e., $K$ has finitely many components which are pointwise Whitney regular, 
then the algebraic equality $\C=\Cr$ already implies the equivalence of norms by the open mapping theorem.
However, in the next chapter we will see that $K=\{0\}\cup \{2^{-n}: n\in\N\}$ satisfies $\C=C^1(\R|K)$ although $\C$ is incomplete.
This means that the algebraic equality, in general, does not imply the equivalence of norms. Except for the one-dimensional case, we do not know a characterization of $\C=\Cr$.
Nevertheless, we would like to remark that this property has very poor stability properties. The example of the inward directed cusp mentioned in the introduction is the union of two even convex 
sets whose intersection is an interval (sadly, the two halfs of a broken heart behave better than the intact heart). More surprising is perhaps the following example showing 
that the property $\C=\Cr$ is not stable with respect to cartesian products.

\begin{Ex} 
For $M=\{0\}\cup \{2^{-n}: n\in\N\}$ and $K=M\times [0,1]$ we have $\C\neq\Cr[2]$.
\end{Ex}

\begin{proof}
  We construct a function $f\in \C$ which is equal to $0$ everywhere except for some tiny bumps on the segments $S_n=\{2^{-n}\}\times [0,1]$. More precisely, we fix 
  $\varphi\in C^\infty(\R)$ with support in $[-1,1]$ which is bounded in absolute value by $1$, and satisfies $\varphi(0)=1$. For $(x,y)\in S_n$ we then set 
  $f(x,y)= n^{-3} \varphi( n^2(y-1/n))$. It is easy to check that $f$ is differentiable on $K$ (the only non-obvious point is $(0,0)$ where the derivative is $0$),
  and that one can choose a continuous derivative (because the second partial derivatives on $S_n$ are bounded by $c/n$ 
  where $c$ is a bound for the derivative of $\varphi$). Hence $f\in \C$ but $f\notin \Cr[2]$ because $f$ is not Lipschitz continuous as
  $f(2^{-n},1/n) - f(2^{-n+1},1/n))= n^{-3}$ which is much bigger than the distance between the arguments.
\end{proof}

Let us consider now a topologically regular compact set $K \subseteq \R^d$. We can formulate the main theorem of \cite{Whi34c} in this context as follows.

\begin{Thm}\label{whitney-thm}
Let $K$ be a topologically regular compact set. If $\mathring{K}$ is Whitney regular, then $C^1_{\inte}(K)=\Cr$.
\end{Thm}

Here, we prove that the reverse implication doesn't hold. In \cite{Zob99}, Zobin considers a similar question for Sobolev regularity
where, despite the similarity, the situation is different. For this purpose, we establish the following proposition.

\begin{Prop}\label{Prop:C=Cint}
Let $K$ be a topologically regular compact set and assume that, for all $x \in \partial K$, there exist $C_x >0$ and a neighbourhood $V_x$ of $x$ such 
that each $y \in V_x$ can be joined from $x$ by a rectifiable path in $\mathring{K}\cup\{x,y\}$ of length bounded by $C_x |x-y|$. Then $C^1_{\inte}(K)=\C$.
\end{Prop}

\begin{proof}
Let us take $f \in C^1_{\inte}(K)$, to prove that $f \in \C$ we just have to show the differentiability at $x \in \partial K$. 
For all $y \in V_x$ we get from the remark after the fundamental theorem \ref{thm:pathint} 
\begin{align*}
f(y)-f(x)-\langle df(x),y-x \rangle = \int_\gamma (df-df(x)),
\end{align*} 
where $\gamma$ is as stated in the assumptions. This is enough to get the differentiability at $x$, as we 
did previously in proposition \ref{prop:whba}.
\end{proof}

We now construct a topologically regular compact connected set whose interior is not Whitney regular but where equality $C^1_{\inte}(K)=\Cr$ holds. 

\begin{Ex}
Let $\Omega$ be the open unit disk in $\R^2$ from which we remove, as in the Sauter's example, tiny disjoints balls which accumulate at $S=\{0\} \times [-\frac{1}{2},\frac{1}{2}]$. 
Then $K=\overline{\Omega}$ is connected, topologically regular and Whitney regular (by the same argument as explained below). In particular, from theorem \ref{Thm:E=C}, we know that $C^1(\R^2 | K)= \C$.

Of course, $\mathring{K}$ is not Whitney regular, because $S$ is not contained in $\mathring{K}$, but 
the assumptions of proposition \ref{Prop:C=Cint} are satisfied and hence $\C=C^1_{\inte}(K)$: Indeed, a boundary point $x$ of $K$ is either a 
boundary point of the unit disc or of one of the tiny removed discs in which cases the condition is clear, or $x$ is on the segment $S$. If then $y$ is any other point of $K$
we connect it by a short path to a point $\tilde y\in \mathring K$, consider the line from $\tilde y$ to $x$ and, whenever this line intersects one of the removed discs, we replace this intersection by a path through $\mathring K$ which is parallel to 
the boundary of the little disc. The total length increase of this new path is by a factor $\pi+\varepsilon$.
\end{Ex}

To give a partial converse of Whitney's theorem \ref{whitney-thm} we state the following consequence of \ref{thm:compl}.
\begin{Prop}
Let $K$ be a topologically regular compact set. If $C^1_{\inte}(K)=\C$ (in particular, if
$C^1_{\inte}(K)=\Cr$ holds), then $K$ has only finitely many connected components which are all pointwise Whitney regular.
\end{Prop}

\begin{proof}
 If  $C^1_{\inte}(K)=\C$, then $(\C,\|\cdot \|_{\C})$ is complete and hence theorem \ref{thm:compl} implies the stated properties of $K$.
\end{proof}

\section{The one-dimensional case}

In this last section we completely characterize the equality between the three spaces of $C^1$-functions for compact subsets of $\R$.  
Of course, all three spaces coincide for compact sets with only finitely many components, and otherwise $\C$ is incomplete by proposition \ref{Prop:incic}
and thus different from $C^1_{\inte}(K)$. The remaining question when $\C=\Cro$ will depend on the behaviour of the bounded connected components 
of $\R\setminus K$ which we call {\it gaps of $K$}. These are thus maximal bounded open intervals $G$ in the complement, and we denote their length by $\ell(G)$.

The simple idea is that small gaps are dangerous for the 
Lipschitz continuity on $K$ which is a necessary condition for $C^1$-extendability. In fact, we will show that $\C\neq \Cro$ whenever there are $\xi\in K$ and
nearby gaps of $K$ of length much smaller than the distance of the gap to $\xi$. To be precise, we define, for positive $\eps$,
\[
  \sigma_\eps(\xi)=\sup\left\{\frac{\sup\{|y-\xi|: y\in G\}}{\ell(G)}: G\subseteq (\xi-\eps,\xi+\eps) \text{ is a gap of $K$}\right\}
\]
where $\sup \emptyset=0$. Of course, these $[0,\infty]$-valued functions are increasing with respect to $\eps$ and thus we can define the {\it gap-structure function}
\[
  \sigma(\xi)=\lim_{\eps\to 0} \sigma_\eps(\xi).
\]

\begin{Thm}
For a compact set $K \subseteq \R$ we have  $\C=\Cro$ if and only if $\sigma(\xi) < \infty$ for all $\xi \in K$.
\end{Thm}

Before giving the proof let us discuss some examples. The Cantor set $K$ satisfies $\sigma(\xi)=\infty$ for all $\xi\in K$ so that $C^1(K)\neq C^1(\R|K)$.

Other simple examples are sets of the form $K=\{0\}\cup\{x_n:n\in\N\}$ for decreasing sequences $x_n\to 0$. Then $\sigma(x_n)=0$ for all $n\in\N$ and only the behaviour
of $\sigma(0)$ depends on the sequence. Since the gaps of $K$ are $(x_{n+1},x_n)$ we get $\sigma(0)=\limsup \frac{x_n}{x_n-x_{n+1}}$.
This is finite for fast sequences like $x_n=a^{-n}$ with $a>1$ but infinite for slower sequences like $x_n=n^{-p}$ for $p>0$.

This class of examples can be easily modified to topologically regular sets of the form $K=\{0\}\cup\bigcup_{n\in\N} [x_n,x_n+r_n]$. For $r_n=e^{-2n}$ we get 
$\sigma(0)<\infty$, e.g., for $x_n=e^{-n}$ and $\sigma(0)=\infty$ for $x_n=1/n$.

\begin{proof}
  We will use Whitney's \cite{Whi34b} characterization that $f\in \Cro$ if and only if, for all non-isolated $\xi\in K$,
  \[
    \lim_{x,y\to \xi} \frac{f(x)-f(y)}{x-y} =f'(\xi).
  \]
Let us first assume $\sigma(\xi)=\infty$ for some $\xi\in K$. There is thus a sequence of gaps $G_n=(a_n,b_n)\subseteq (\xi-1/n,\xi+ 1/n)$ with
$\sup\{|y-\xi|: y\in G_n\}/|a_n-b_n| >2n$. Passing to a subsequence, we may assume that all these gaps are on one the same side of $\xi$, say $\xi < a_n <b_n$, so that
$b_n-\xi > 2n (b_n-a_n)$. 

Moreover, again by passing to a subsequence and using $\sigma_\eps(\xi)=\infty$ for $\eps=(b_n-a_n)/2$, we can reach $b_{n+1}<a_n$ and that the 
midpoints $y_n=(a_n+b_n)/2$ of the gaps satisfy
\[
  \frac{y_n-y_{n+1}}{b_n-a_n} \ge n.
\]
We now define $f: K\to\R$ by $f(x)=(y_n-\xi)/n$ for $x\in K\cap (y_n,y_{n-1})$ (with $y_0=\infty$) and $f(x)=0$ for $x\le\xi$. Since the jumps of $f$ are outside
$K$ it is clear that $f$ is differentiable at all points $x\in K\setminus \{\xi\}$ with $f'(x)=0$. To show the differentiability at $\xi$ with $f'(\xi)=0$ we calculate for 
$x\in K\cap (y_n,y_{n-1})$
\[
  \left|\frac{f(x)-f(\xi)}{x-\xi}\right| = \left| \frac{(y_n-\xi)/n}{x-\xi}\right| \le \left|\frac{(y_n-\xi)/n}{y_n-\xi}\right| \le \frac 1n.
\]
Thus, $f\in C^1(K)$ but $f\notin \Cro$ because
\[
  \frac{f(b_n)-f(a_n)}{b_n-a_n}=\frac{(y_n-\xi)/n -(y_{n+1}-\xi)/(n+1)}{b_n-a_n} 
 \ge \frac{(y_n -y_{n+1})/n}{b_n-a_n} \ge 1.
\]

Let us now assume $\sigma(\xi)<\infty$ for all $\xi\in K$. To prove that every $f\in \C$ belongs to $\Cro$ we first show that we can assume $f'=0$. Indeed,
we extend $f':K\to\R$ to a continuous function $\varphi:\R\to\R$ and consider $g(x)=f(x)-\int_0^x \varphi(t) dt$. Then $g\in \C$ satisfies $g'=0$ and $g\in \Cro$ implies $f\in \Cro$.

Let us thus fix $f\in \C$ with $f'=0$. We have to show Whitney's condition stated above at any non-isolated point $\xi$ which, for notational convenience, we may assume to be $\xi=0$.
We fix $c>\max\{\sigma(0),1\}$ and $\eps\in(0,1)$.
There is thus $\delta>0$ such that, because of the differentiability at $\xi=0$ with $f'(0)=0$, we have 
\begin{align}\label{dq}
  \left|\frac{f(x)-f(0)}{x-0}\right| <\frac{\eps}{2c}
\end{align}
for all $x\in K$ with $|x|<\delta$ and, because of $\sigma_\delta(\xi)<c$ for small enough $\delta$, 
\[   
  \sup\{ |y|: y\in G\} \le c \ell(G)
\]
for all gaps $G\subseteq (-\delta,\delta)$.
For $x,y\in K\cap (-\delta,\delta)$ we will show 
\[
  \left|\frac{f(x)-f(y)}{x-y}\right| \le \eps.
\]
If $x,y$ are in the same component of $K$ this quotient is $0$ because $f$ is locally constant.
Moreover, if $x,y$ are on different sides of $0$, the quotient is bounded by $\eps$ because of (\ref{dq}) and $c\ge 1$. It remains to consider the case $0<x<y$. 
Then there is a gap $G$ between
$x$ and $y$ and, since $f$ is locally constant, we may decrease $y$ so that $y\in \partial K$ without changing $f(y)$ which thus increases the difference quotient we have to estimate.
This implies that $y$ is the endpoint of gap $G=(a,y)$ with $a\ge x$ which implies 
\[
  |y-x|\ge |y-a| =\ell(G) \ge y/c \ge x/c.
\]
Therefore,
\begin{align*}
  \left|\frac{f(x)-f(y)}{x-y}\right| & \le \left|\frac{f(x)-f(0)}{x-y}\right| +\left|\frac{f(y)-f(0)}{x-y}\right| \\
  &\le c\left|\frac{f(x)-f(0)}{x-0}\right|+c\left|\frac{f(y)-f(0)}{y-0}\right|\le \eps.
  \qedhere
\end{align*}

\end{proof}

\bigskip

\noindent \textbf{Acknowledgement.} Laurent Loosveldt's research was supported by a grant of the FNRS. The present paper was mainly written 
during a research stay of Laurent Loosveldt at Trier University, supported by the University of Liège, the University of the 
Greater Region and another grant of the FNRS.

\bibliographystyle{amsalpha}
\bibliography{C1K}

\end{document}